\documentclass[12pt]{amsart}

\usepackage[usenames]{color}
\input xypic

\usepackage{verbatim}
\usepackage{url}
\usepackage{bm}
\usepackage{pict2e}

\usepackage{amsmath}
\usepackage[centertags]{amsmath}
\usepackage{amsfonts}
\usepackage{amssymb}
\usepackage{amsthm}
\usepackage{newlfont}
\usepackage{url}
\usepackage{bigints}

\theoremstyle{plain}
\newtheorem{thm}{Theorem}

\newtheorem{lem}[thm]{Lemma}
\newtheorem{lem*}[thm]{Lemma}
\newtheorem{prop}[thm]{Proposition}

\theoremstyle{definition}
\newtheorem{dfn}{Definition}
\theoremstyle{remark}
\newtheorem{rem}{Remark}
\newtheorem{rem*}{Remark}

\numberwithin{rem}{section} 
\numberwithin{dfn}{section} 
\numberwithin{equation}{section} 
\numberwithin{thm}{section} 

\def\!{\operatorname{!}}

\def\be{\begin{equation}}
\def\ee{\end{equation}}
\def\beg{\begin{equation*}}
\def\eeg{\end{equation*}}

\def\F{\mathbb F}

\def\1{\bold 1}

\begin{document}
\title[ Linear relations...]{ Linear relations for Lauricella $F_D$  functions and symmetric polynomials}

\author{Piotr Kraso{\'n}, Jan Milewski}

\date{\today}

\address{  Institute of Mathematics, Department of Exact and Natural Sciences, University of Szczecin, ul. Wielkopolska 15, 70-451 Szczecin, Poland 
}
\email{ piotrkras26@gmail.com}

\address{Institute of Mathematics, Faculty of Electrical Engineering, Pozna{\'n} University of Technology, ul. Piotrowo 3A, 60-965 Pozna{\'n}, Poland}
\email{jsmilew@wp.pl}

\thanks{}

\maketitle

\begin{abstract}
In this paper we develop an algorithm   for  obtaining some new  linear relations among the Lauricella $F_D$ functions. 
Relations we obtain, generalize those hinted in the work of B. C. Carlson \cite{c63}.
The coefficients of these relations are contained in the ring of  polynomials in the variables $x_1,\dots,x_N$  or  in some exceptional cases in the field of rational functions ${\mathbb R}(x_1,\dots,x_N,p)$.
The method is based on expressing suitably chosen Euler type indefinite integrals 
 associated with these functions recursively as  linear combination of some  other Euler type integrals and elementary functions and then integrating over the interval $[0,1].$ 
We describe the complete algorithm for obtaining these relations. We believe that such relations might be useful in computations.
\end{abstract}
\section{Introduction} 
 Since the  hypergeometric functions in many variables are very general, relations among them drawn a lot of attention
cf. \cite{c63}, \cite{ks03}, \cite{kt11}, \cite{v03}, \cite{ex83}, \cite{b08}. 
In his seminal paper on the Lauricella $F_D$ function \cite{c63} B.C. Carlson introduced some linear, with the coefficients in the ring of homogenous polynomials, relations among the  hypergeometric functions $R(a,b_1,\dots , b_N, z_1,\dots , z_N).$ These functions $R(a,b_1,\dots , b_N, z_1,\dots , z_N)$, defined in \cite{c63}, are directly related to the Lauricella functions $F_D^{(N)}\left(\begin{matrix}a,&b_1,\dots , b_N\\  & c  \end{matrix}\mid x_1,\dots , x_N \right)$ where $c=b_1+\dots+b_N\neq 0.$ 
  The author of \cite{c63}   obtained contiguous relations for the functions $R.$ These contiguous relations require change of two parameters by one (changing $b_i$ by one implies  the corresponding change of $c$) 
 and generalize the well-known relations for the Appell functions \cite{ak26}. The author of \cite{c63} claims that using the recursion based on the \cite[formula (4.6)]{c63}  and solution of some system of equations one can obtain 
 formulas in which  only one parameter changes e.g. the formulas involve either $a+n, n=0,\pm1, \pm 2,\dots$ or the change of one of the coefficients $b_i$. See the discussion at the end of section IV of \cite{c63}.
 Naturally, in the case of the Lauricella function $F_D$ a change in the parameter $a$ requires a corresponding change 
 in the parameter $c$ (cf.  (\ref{Lur1})).
 In  this paper we derive  (cf. formulas (\ref{relpol}) - (\ref{relpol01}) ) such "pure" relations  by means of a direct algorithm involving calculation of carefully chosen
 Euler type integrals that are associated with those that represent the Lauricella $F_D$ functions. Direct means that we work   with the Lauricella functions not with the other hypergeometric function (as $R$ or other) and also that we do not have to use recursion repetitively. Moreover, we do not have to limit ourselves to the condition $c=b_1+\dots +b_N\neq 0.$ In our method all the parameters have to obey only the usual conditions for the convergence of the defining power series. Therefore, some of our relations are new and generalize 
 those hinted in \cite{c63}.
 We provide explicit formulas for the coefficients of our linear relations for the Lauricella $F_D$ functions. They are described by means of the symmetric polynomials in the arguments of the functions $F_D.$
We describe a complete algorithm of symbolic computations that leads  to the linear     
  relations among the Lauricella $F_D$ functions. The coefficients of the relations are either in the ring of  polynomials or the  field of rational functions
  in several variables (c.f. Theorem \ref{main})
  The methods described in this paper are valid for arbitrary number of variables $N$ i.e. for any $F_D^{(N)}\left(\begin{matrix}a,&b_1,\dots , b_N \\  & c  \end{matrix}\mid x_1,\dots , x_N \right).$

\section{Preliminaries} 
Let 
\begin{equation}\label{Lur1}
F_D^{(N)}\left(\begin{matrix}a,&b_1,\dots , b_N \\  & c  \end{matrix}\mid x_1,\dots , x_N \right) = {\sum}_{i_1,\dots , i_N = 1}^{\infty}\frac{(a)_{i_1+\dots +i_N} {{(b_{1})}_{i_1}}\dots 
 {({b_{N})}_{i_N}}}{(c)_{i_1+\dots +i_N}\cdot {i_1}!\cdot \dots  {i_N}!}x_1^{i_1}\dots \, x_N^{i_N},
\end{equation}
where  $(s)_{i}$ is a Pochhammer symbol,
denote the Lauricella function od type $D$  in $N$ variables  \cite{l93}. Then  for $|x_i|<1, i=1,\dots ,N$ and   $c>a>0$ we have an Euler type representation:
\begin{equation}\label{Lur2}
F_D^{(N)}\left(\begin{matrix}a,& b_1, \dots , b_N\\  & c \end{matrix} \mid x_1,\dots , x_N\right) = L{\int}_{0}^{1} t^{a-1}(1-t)^{c-a-1}(1-x_1t)^{-b_1}\dots (1-x_Nt)^{-b_N}dt
\end{equation}
where $L=\frac{\Gamma (c)}{{\Gamma (a)} {\Gamma (c-a)}}$ and $\Gamma$ denotes the gamma Euler  function. Notice that $L^{-1}=B(a,c-a)$ where 
\[ B(x,y)={\int}_0^1t^{x-1}(1-t)^{y-1}dt,\]
defined for  ${\mathrm{Re}}x>0$ and ${\mathrm{Re}}y >0 $  is the integral representation of the beta Euler function.

Let us express  the function under the integral in the right hand side of (\ref{Lur2}) in the following form:
\begin{equation}\label{integralL}
K(t)=t^{{\beta}_1}(1-t)^{{\beta}_2}{\prod}_{i=1}^{N}(1-x_{i}t)^{{\alpha}_i}
\end{equation}
where $ {\beta}_1=a-1, {\beta}_2=c-a-1$ and
${\alpha}_i=-b_{i}$ for  $ i=1,\dots, N.$ 
Let further
\begin{equation}\label{intj}
J({\beta}_1,{\beta}_2,{\alpha}_1,\dots {\alpha}_n)={\int}_0^1t^{{\beta}_1}(1-t)^{{\beta}_2}{\prod}_{i=1}^{N}(1-x_{i}t)^{{\alpha}_i}dt
\end{equation}

\begin{rem}
Notice that the integral in the right-hand side of  (\ref{Lur2}) is convergent for   \linebreak ${\mathrm{Re}}x_i <1, i=1,\dots ,n.$ 
Therefore the formula (\ref{Lur2}) describes in fact an analytic continuation of the function given by the formula 
(\ref{Lur1}).
\end{rem}
In what follows we  consider the real integrals. The modification for the complex variables is standard.

\subsection{Some elementary identities}
In this subsection we show how some well known  elementary relations for the Gauss'  hypergeometric functions readily generalize to the case of Lauricella  $F_D$ functions in $n$-variables. 
\begin{prop}\label{elm1} We have the following identity (the first Pfaff identity):
\begin{equation}\label{fr}
F_D^{(N)}\left(\begin{matrix}a,&b_1,\dots , b_N \\  & c  \end{matrix}\mid x_1,\dots , x_N \right) =\end{equation}
\[ \left[ \prod_{i=1}^{N}(1-x_i)^{-b_i}\right]F_D^{(N)}\left(\begin{matrix}c-a,&b_1,\dots , b_N \\  & c  \end{matrix}\mid y_1,\dots , y_N \right)  \] 
where
\[ y_i=\frac{-x_i}{1-x_i} \; .\]
\end{prop}
\begin {proof}
Substituting $u=1-t$ in the integral (\ref{intj}) we obtain the equality
\begin{equation}\label{ifr} J({\beta}_1, {\beta}_2 ; \alpha _1, \ldots , \alpha_n; x _1, \ldots , x_n)=\prod_{i=1}^n (1-x_i)^{\alpha _i} J(\beta_2, \beta_1; \alpha _1, \ldots , \alpha_n; y _1, \ldots , y_n),  \end{equation}
 \[ y_i=\frac{-x_i}{1-x_i} \; .\]
(\ref{fr}) follows from (\ref{ifr}).
 
\end{proof}

\begin{prop}\label{elm2}
 We have the following identity (the second Pfaff identity):
\begin{equation}\label{f2}
F_D^{(N)}\left(\begin{matrix}a,&b_1,\dots , b_N \\  & c  \end{matrix}\mid x_1,\dots , x_N \right) =\end{equation} 
\[(1-x_i)^{-a}F_D^{(N)}\left(\begin{matrix} a,& c-\sum b_i, b_1,\dots , \widehat b_i, \ldots , b_N \\  & c  \end{matrix}\mid y_i, z_{i,1},\dots ,\widehat{z_{i,i}}, \ldots z_{i,N} \right) \]

where
\[ z_{i,j}=\frac{x_j-x_i}{1-x_i} \; ,\] 
$y_i$ is as in the Proposition \ref{elm1} and hat over the symbol means that corresponding variable is omitted.

\end{prop}
\begin {proof}
We use the following rational substitution:
\begin{equation}\label{Pfaff2}  v=\frac{(1-x)t}{1-xt}, \quad dv =\frac{(1-x)dt}{(1-xt)^2} \end{equation}
which yields the following identity:
\be\label{j2} J(\beta_1, \beta_2 ; \alpha _1, \ldots , \alpha_n; x _1, \ldots , x_n)= \ee
\[ \frac{1}{{(1-x_i)}^{\beta_1+1}} J\left(\beta_1, \beta_2; -(\beta_1+\beta_2+\sum \alpha _j+2),
\alpha _1, \ldots , \widehat \alpha _i \ldots  \alpha_n; y _i, z_{i,1} \ldots , \widehat z_{ii} , \ldots z_{i,n}\right),  \]
 \[ z_{i,j}=\frac{x_j-x_i}{1-x_i} \; .\]
( \ref{f2}) follows directly from (\ref{j2}).
\end{proof}

\begin{prop} We have the following identities:
\begin{equation}\label{first}
aF_D^{(N)}\left(\begin{matrix}a+1,&b_1,\dots , b_N \\  & c+1  \end{matrix}\mid x_1,\dots , x_N \right) +
(c-a)F_D^{(N)}\left(\begin{matrix}a,&b_1,\dots , b_N \\  & c+1  \end{matrix}\mid x_1,\dots , x_N \right) =\end{equation}
\[
cF_D^{(N)}\left(\begin{matrix}a,&b_1,\dots , b_N \\  & c  \end{matrix}\mid x_1,\dots , x_N \right), \] 

\begin{equation}\label{second}
ax_iF_D^{(N)}\left(\begin{matrix}a+1,&b_1,\dots , b_N \\  & c+1  \end{matrix}\mid x_1,\dots , x_N \right) +
cF_D^{(N)}\left(\begin{matrix}a,&b_1,\dots ,b_i-1, \ldots b_N \\  & c  \end{matrix}\mid x_1,\dots , x_N \right)= 
\end{equation}
\[
cF_D^{(N)}\left(\begin{matrix}a,&b_1,\dots , b_N \\  & c  \end{matrix}\mid x_1,\dots , x_N \right), \] 

\begin{equation}\label{third}
apF_D^{(N)}\left(\begin{matrix}a+1,&b_1,\dots , b_N \\  & c+1  \end{matrix}\mid x_1,\dots , x_N \right) +
cF_D^{(N)}\left(\begin{matrix}a,&b_1,, \ldots b_N , -1\\  & c  \end{matrix}\mid x_1,\dots , x_N , p\right) =\end{equation}
\[
cF_D^{(N)}\left(\begin{matrix}a,&b_1,\dots , b_N \\  & c  \end{matrix}\mid x_1,\dots , x_N \right)\qquad p\notin \{ 1, x_1, \ldots x_N\}. \]

\end{prop}
\begin {proof}
Multiplying the expression under  the integral sign in (\ref{intj}) by $[t+(1-t)]$ yields the following identity:
\begin{equation}\label{fj1} J(\beta_1+1, \beta_2 ; \alpha _1, \ldots , \alpha_n; x _1, \ldots , x_n) +J(\beta_1, \beta_2+1 ; \alpha _1, \ldots , \alpha_n ; x _1, \ldots , x_n) \end{equation}
\[=J(\beta_1, \beta_2 ; \alpha _1, \ldots , \alpha_n; x _1, \ldots , x_n). \]
(\ref{first}) follows from (\ref{fj1}).

Similarly multiplying the expression under the integral sign in the formula (\ref{intj}) by $[x_it+(1-x_it)]$ (resp. $[pt+(1-pt)]$)
leads to the formula (\ref{second}) (resp. (\ref{third})).

\end{proof}

\begin{prop} We have the following identities:
\begin{equation}\label{reldiff}
  (c-1)F_D^{(N)}\left(\begin{matrix}a-1,&b_1,\dots , b_N \\  & c-1  \end{matrix}\mid x_1,\dots , x_N \right) -\end{equation}
\[(c-1)F_D^{(N)}\left(\begin{matrix}a,&b_1,\dots , b_N \\  & c-1  \end{matrix}\mid x_1,\dots , x_N \right) +\] 
\[  \sum_j b_jx_j
F_D^{(N)}\left(\begin{matrix}a,&b_1,\dots , b_j+1, b_N \\  & c  \end{matrix}\mid x_1,\dots , x_N \right)=0. \] 
\end{prop}
\begin {proof}
Differentiating the expression under  the integral sign in (\ref{intj}) one obtains the following identity:

\begin{equation}\label{diff} \beta_1 J(\beta_1-1, \beta_2 ; \alpha _1, \ldots , \alpha_n; x _1, \ldots , x_n)-\beta_2 J(\beta_1, \beta_2-1 ; \alpha _1, \ldots , \alpha_n; x _1, \ldots , x_n)-  \end{equation}
\[ \sum _j \alpha _j x_j J(\beta_1, \beta_2 ; \alpha _1, \ldots ,\alpha_j-1, \ldots  \alpha_n; x _1, \ldots , x_n)=0. \] 

(\ref{Lur2}) and (\ref{diff}) give (\ref{reldiff}).
\end{proof}

\section{Deriving  relations for Lauricella  $F_D$ functions}

In this section we derive the relations described in the Introduction. These relations involve changes of one exponent in the integral representation (2.2).

Let $K(t)$ be given by (\ref{integralL}) and
\be\label{P}
P(t)=t(1-t){\prod}_{i=1}^{N}(1-x_it), \quad K_1(t)=P(t)K(t).
\ee

Consider the integrals of the form:
\be\label{intp}
I_{n}={\int}_0^1t^nK(t)dt  \qquad  n\in {\mathbb Z},  \quad n\geq 0
\ee
and 
\be\label{intt}
I_{n,p}={\int}_0^1 u^nK(t)dt, \quad  n\in {\mathbb Z}, 
\ee
where $u=1-pt.$

Observe that 
\[ I_0=J({\beta}_1,{\beta}_2,{\alpha}_1,\dots {\alpha}_n) \]
for $ {\beta}_1=a-1, {\beta}_2=c-a-1$ and
${\alpha}_i=-b_{i}$. 

The integrals (\ref{intp}) and (\ref{intt}) are closely related to the Lauricella functions $F^{k}_D$ as from the integral 
representation (\ref{Lur2}) we see that:

\begin{enumerate}
\item[]
\be\label{in} I_n=B(a+n, c-a) F_D^{(N)}\left(\begin{matrix}a+n,&b_1,\dots , b_N \\  & c+n  \end{matrix}\mid x_1,\dots , x_N \right) \ee
for $c>a,$

\item[]
\be I_{n, p}=B(a, c-a) F_D^{(N+1)}\left(\begin{matrix}a,&b_1,\dots  b_N, -n \\  & c  \end{matrix}\mid x_1, \ldots x_N, p \right) \ee
for $p\neq  x_1, \ldots , x_N, p<1,$

\item[]

\be I_{n,1}=B(a, c+n-a) F_D^{(N)}\left(\begin{matrix}a,&b_1,\dots , b_N \\  & c+n  \end{matrix}\mid x_1,\dots , x_N \right) \ee
for $c+n>a,$
\item[]
\be\label{in1} I_{n, x_i}=B(a, c-a) F_D^{(N)}\left(\begin{matrix}a,&b_1,\dots ,b_i-n, \ldots , b_N \\  & c  \end{matrix}\mid x_1, \ldots x_N \right). \ee

\end{enumerate}

\subsection{Relevant derivatives}

The starting point for obtaining the desired identities for the Lauricella function $F_D^N$ are two suitably chosen derivatives. 
The first one is the following:

\be\label{poly1}
\frac{d}{dt}[t^nK_1(t)]=[nt^{n-1}P(t)+t^nW(t)]K(t), \ee 
where $P(t)$ is given by (\ref{P}) and

\be\label{pw1} 
 W(t)=({\beta}_1+1)(1-t){\prod}_{i=1}^{N}(1-x_{j}t) -\ee
\[ ({\beta}_2+1)t {\prod}_{j=1}^{N}(1-x_{j}t) -\sum_l ({\alpha}_l+1)x_lt(1-t){\prod}_{j\neq l}^{N}(1-x_{j}t)\]

Express $P(t)$ and $W(t)$ in the following way:
\be\label{PW1}
P(t)={\sum}_{k=1}^{N+2}d_{k}t^{k}, \quad W(t)={\sum}_{k=0}^{N+1}e_{k}t^{k}
\ee \label{wed}
\[ d_k=d_k(x_1,\ldots , x_N), \quad  e_k=e_k(x_1,\ldots , x_N)\; . \]

The second considered derivative is similar to the first one:
\be\label{np1}
\frac{d}{dt}[u^nK_1(t)]=[-npu^{n-1}P(t)+u^nW(t)]K(t),    \quad u=(1-pt)\; .\ee  

Representing $P(t)$ and $W(t)$ by means of the variable $u=1-pt$ we obtain the following formulas:
\be\label{ttt} t=-\frac{1}{p}(u-1), \quad 1-t=\frac{1}{p}(u+p-1), \quad 1-x_jt=\frac{x_j}{p}\left(u+\frac{p-x_j}{x_j}\right), \;  \ee
\be\label{PPP} P(t)=\frac{-x_1\cdot \ldots \cdot x_N}{p^{N+2}}\left[ (u-1)(u+p-1) \prod _{j=1}^N \left( u+\frac{p-x_j}{x_j}\right)\right] . \ee

Thus

\begin{align}\label{wup} 
 W(t)= \frac{1}{p^{N+1}} x_1\cdot \ldots \cdot x_N 
 \cdot \left[(\beta_1+1)(u+p-1) \prod _{j=1}^N \left( u+\frac{p-x_j}{x_j}\right) +\right. \\
\left (\beta_2+1) (u-1) \prod _{j=1}^N \left( u+\frac{p-x_j}{x_j}\right) +\right.\nonumber\\
 \left. \sum_l(\alpha_l+1)x_l (u-1)(u+p-1) \prod _{j\neq l}\left( u+\frac{p-x_j}{x_j}\right)  \right]\nonumber .
\end{align}

Let
\be \label{np11} P(t)= \sum_{k=0}^{N+2} d_k (p) u^k, \quad   W(t)=\sum_{k=0}^{N+1} e_k (p)u^k ,\ee
\[ d_k(p)=d_k(p, x_1,\ldots , x_N), \quad  e_k(p)=e_k(p, x_1,\ldots , x_N) .\]

We have the following
\begin{thm}\label{main}
The following identities hold true:
\begin{enumerate}
\item[$\bullet$]
for the integrals $I_n$
\be \label{calbp} \sum _{k=0}^{N+1} (nd_{k+1} +e_k) I_{k+n}=0, \ee
\item[$\bullet$]
for the integrals $I_{n, p}$ where $p\neq x_1, \ldots , x_N, p<1$
\be\label{calpol}
{\sum}_{k=0}^{N+2}[-npd_k(p)+e_{k-1}(p)]I_{k+n-1,p}=0,
\ee
\item[$\bullet$]
for the integrals $I_{n, p}$ where $p=1, x_1, \ldots , x_N$
\be \label{calpol0}  \sum _{k=0}^{N+1} (-npd_{k+1}(p) +e_k(p)) I_{k+n,p}=0. \ee
\end{enumerate}
The expressions
\[ d_k=d_k(x_1, \ldots x_N), \quad  e_k=e_k(x_1, \ldots x_N) \]
\[ d_k(p)=d_k(p, x_1, \ldots x_N), \quad  e_k(p)=e_k(p, x_1, \ldots x_N) \]
are the coefficients of expansions (\ref{PW1}) and (\ref{np11}).
They are polynomial in  $x_1, \ldots x_N$, and rational in  $p.$
\end{thm}
\begin{proof}
The identities  (\ref{calbp})-(\ref{calpol})  follow from the equalities (\ref{poly1}),   (\ref{PW1}), (\ref{np1}) and (\ref{np11}). 
Integrating both sides of  (\ref{poly1}) and (\ref{np1}) from $0$ to $1$ and taking into account that $K(0)=K(1)=0$ we obtain 
the above identities. The dependence  of the coefficients 
on  $x_1, \ldots x_N$ and $p$  will become   clear in the next section, where we  give the precise algorithm of determining them.
\end{proof}
\begin{thm}\label{main1}
Let ${\beta}_1=a-1,\, {\beta}_2=c-a-1$ and ${\alpha}_i=-b_i.$
The following identities for the Lauricella $F_D$ functions hold true:
\begin{enumerate}
\item[$\bullet$] for $n\ge 0$ and $c>a>0$
\begin{multline} \label{relbp} \sum _{k=0}^{N+1} (nd_{k+1} +e_k)B(a+k+n,c-a)\cdot \\F_D^{(N)}\left(\begin{matrix}a+k+n,&b_1,\dots , b_N \\  & c+k+n  \end{matrix}\mid x_1,\dots , x_N \right)=0, \end{multline}
where the coefficients $d_k$  ($e_k$ resp.) are given by Lemma \ref{ak} (Lemma \ref{bk} resp.).

\vspace{3mm}

\item[$\bullet$]
for $n<0,$ $c>a>0$ and $p\neq x_1, \ldots , x_N, p<1$
\begin{multline}\label{relpol}
{\sum}_{k=0}^{N+2}[-npd_k(p)+e_{k-1}(p)] F_D^{(N+1)}\left(\begin{matrix}a,&b_1,\dots  b_N, 1-k-n \\  & c  \end{matrix}\mid x_1, \ldots x_N, p \right)=0,
\end{multline}
where the coefficients $d_k(p)$  ($e_k(p)$ resp.) are given by Lemma \ref{ak1} (Lemma \ref{bk1} resp.).

\vspace{3mm}

\item[$\bullet$]
for $p=1$ and $c+n>a>0$
\begin{multline} \label{relpol0}  \sum _{k=0}^{N+1} (-npd_{k+1}(p) +e_k(p)) B(a, c+k+n-a)\cdot \\F_D^{(N)}\left(\begin{matrix}a,&b_1,\dots , b_N \\  & c+n+k  \end{matrix}\mid x_1,\dots , x_N \right) =0. \end{multline}
where the coefficients $d_k(p)$  ($e_k(p)$ resp.) are given by Lemma \ref{a11} (Lemma \ref{bk11} resp.).

\vspace{3mm}

\item[$\bullet$]
for $p=x_1,\dots, x_N$, $c>a$ 
\begin{multline} \label{relpol01}  \sum _{k=0}^{N+1} (-npd_{k+1}(p) +e_k(p)) \cdot F_D^{(N)}\left(\begin{matrix}a,&b_1,\dots, b_i-n-k,\dots , b_N \\  & c  \end{matrix}\mid x_1,\dots , x_N \right) =0. \end{multline}
where the coefficients $d_k(p)$  ($e_k(p)$ resp.) are given by Lemma \ref{akxi} (Lemma \ref{bk1xi} resp.).

\end{enumerate}

\end{thm}
\begin{proof}
Substitute  formulas (\ref{in})-(\ref{in1}) into the corresponding formulas of Theorem \ref{main}.
\end{proof}

\section{Algorithm for determining coefficients.}

In this section   we  describe coefficients $d_k, d_k(p),\,\, k=1,\dots N+2$ and 
$e_k, e_k(p), \,\, k=0,\dots ,N+1$  as functions of the variables $x_1,\dots ,x_N, p.$
\begin{dfn}
By ${\sigma}_l(y_1, \dots , y_k)$ we will denote the $l$-th elementary symmetric polynomial in $k$ variables. 
We assume that ${\sigma}_0(y_1, \dots , y_k):=1$ and ${\sigma}_l(y_1, \dots , y_k):=0$ if $l>k$ or $l< 0$ i.e.
\be\label{els}
 {\sigma}_l(y_1, \dots , y_k)=\begin{cases}
 {\sum}_{1\leq {i_1}< \dots < {i_l}\leq k}y_{i_1}\cdot\dots\cdot y_{i_l} \qquad\hfill 0<l\leq k\\
 1\hfill l=0, \, k\geq 0 \\
 0 \hfill l<0 \quad {\mathrm{or}} \quad l>k.
 \end{cases}
 \ee 
 Notice that for $k=0$ the set of variables is empty and we adopt the convention that in this case ${\sigma}_0=1$ as well.
\end{dfn}
The following well-known  lemma will be useful:
\begin{lem}\label{symmetric}
\be\label{symm1}
{\sigma}_l(y,y_1, \dots , y_k)=y{\sigma}_{l-1}(y_1, \dots , y_k) + {\sigma}_l(y_1, \dots , y_k).
\ee\qed

\end{lem}

\subsection{Determination of the coefficients $d_k$, $e_k$.}

\begin{lem}\label{ak}
The coefficients $d_k=d_k(x_1,\dots ,x_N)$ are given by the following formulas:
\begin{equation*}
 d_k
=
(-1)^{k-1}[{\sigma}_{k-1}(x_1,\dots ,x_N)+{\sigma}_{k-2}(x_1,\dots, x_N)], \\
 \hfill k=0,\dots N+2\\
\end{equation*}
\end{lem}
\begin{proof}
Follows from (\ref{P}) and the equality $(1-t){\prod}_{i=1}^N(1-x_it)={\sum}_{i=0}^{N+1}(-1)^i{\sigma}_i(1, x_1,\dots ,x_N)t^i$
and Lemma \ref{symmetric}.
\end{proof}

\begin{lem}\label{bk}
The coefficients $e_k=e_k(x_1,\dots, x_N),\,\, k=0,\dots N+1$ are given by the following formulas:

\be e_k= ({\beta}_1+1) v_k+({\beta}_2+1)w_k+\sum_l ({\alpha}_l+1)x_lz_{k,l} ,\ee
where
\[ v_k=(-1)^{k}[{\sigma}_{k}(x_1,\dots ,x_N)+{\sigma}_{k-1}(x_1,\dots, x_N)],\]
\[ w_k=(-1)^{k}{\sigma}_{k-1}(x_1,\dots, x_N),\]
\[ z_{k,l}=(-1)^{k}[{\sigma}_{k-1}\left( (x_i)_{i\neq l}\right)+{\sigma}_{k-2}\left( (x_i)_{i\neq l}\right)],\]

\end{lem}
\begin{proof}
Follows from (\ref{poly1}) and similar  as in Lemma \ref{ak} computation.
\end{proof}

\subsection{Determination of the coefficients $d_k (p)$ and $e_k (p)$}

\begin{lem}\label{ak1}
The coefficients $d_k(p)=d_k(p,x_1,\dots ,x_N)$ in formulas (\ref{np11}) where $u=1-pt$ are given by the following formulas:
\be\label{akp}
d_k(p)= \frac{x_1\cdot \ldots \cdot x_N}{p^{N+2}}\left[\sigma _{N+1-k}\left( p-1;\left(\frac{p-x_j}{x_j}\right) _{j=1,\ldots , N}\right) \right. \ee
\[ -\left. \sigma _{N+2-k}\left(p-1;\left(\frac{p-x_j}{x_j}\right) _{j=1,\ldots , N}\right) \right] , \quad k=0, \ldots , N+2 \]
\end{lem}
\begin{proof}
From (\ref{PPP}) we get 
\be d_k(p)=\frac{-x_1\cdot \ldots \cdot x_N}{p^{N+2}}\sigma _{N+2-k}\left(-1, p-1;\left(\frac{p-x_j}{x_j}\right) _{j=1,\ldots , N}\right) .\ee
Now apply Lemma \ref{symmetric} with $y=-1.$
\end{proof}

\begin{lem}\label{bk1}
The coefficients $e_k(p,x_1,\dots ,x_N)$ in formulas (\ref{np11}) where $u=1-pt$ are given by the following formulas:
\be\label{bbkk} e_k(p)=\frac{1}{p^{N+1}} x_1\cdot \ldots \cdot x_N \cdot  \left[ (\beta_1+1)v_k(p)+(\beta_2+1)w_k(p)+\sum_l(\alpha_l+1)x_l z_{k,l}(p) \right]\ee
where
\[ v_k(p)= \sigma _{N+1-k} \left(p-1, \left(\frac{p-x_j}{x_j}\right)_{j=1, \ldots ,N} \right)  ,\]
\[w_k(p)= \sigma _{N+1-k} \left(-1, \left(\frac{p-x_j}{x_j}\right)_{j=1, \ldots ,N} \right) , \]
\[ z_{k,l}(p)=\sigma _{N+1-k}  \left(-1, p-1,  \left(\frac{p-x_j}{x_j}\right)_{j\neq l} \right).\]
\end{lem}
\begin{proof}
Follows from (\ref{wup}) and  the fact that ${\prod}_{i=1}^M(u+c_i)={\sum}_{k=0}^M{\sigma}_{M-k}(c_1,\dots ,c_M)u^k.$
\end{proof}
\begin{rem}
Notice that using Lemma \ref{symmetric} we can express the coefficients $w_k(p)$ and  $z_{k,l}(p)$ as follows:
\be\label{wkp}
w_k(p)=\sigma _{N+1-k} \left( \left(\frac{p-x_j}{x_j}\right)_{j=1, \ldots ,N} \right)-\sigma _{N-k} \left(\left(\frac{p-x_j}{x_j}\right)_{j=1, \ldots ,N} \right)
\ee
\be
z_{k,l}(p)=\sigma _{N+1-k}  \left( p-1,  \left(\frac{p-x_j}{x_j}\right)_{j\neq l} \right)-\sigma _{N-k}  \left( p-1,  \left(\frac{p-x_j}{x_j}\right)_{j\neq l} \right).
\ee
\end{rem}

\subsection{Determination of coefficients  $d_k (1)$ and $e_k (1)$.}

\begin{lem}\label{a11}
For $p=1$ we obtain:
\be\label{akp1} d_k(1)={x_1\cdot \ldots \cdot x_N}\left[\sigma _{N+1-k}\left( \left(\frac{1-x_j}{x_j}\right) _{j=1,\ldots , N}\right) \right. \ee
\[ -\left. \sigma _{N+2-k}\left(\left(\frac{1-x_j}{x_j}\right) _{j=1,\ldots , N}\right) \right] , \quad k=0, \ldots , N+2 \; .\]
In particular $d_0(x_1,\dots,x_N)=0.$
\end{lem}
\begin{proof}
Put $p=1$ in (\ref{akp}). 
\end{proof}
We have the following specialization of Lemma \ref{bk1}:
\begin{lem}\label{bk11}
For $p=1$  the coefficients $e_k(1)$ are given by the formula (\ref{bbkk}) where
\[ v_k(1)= \sigma _{N+1-k} \left( \left(\frac{1-x_j}{x_j}\right)_{j=1, \ldots ,N} \right)  .\]
\[ w_k(1)=\sigma _{N+1-k} \left( \left(\frac{1-x_j}{x_j}\right)_{j=1, \ldots ,N} \right)-\sigma _{N-k} \left(\left(\frac{1-x_j}{x_j}\right)_{j=1, \ldots ,N} \right) , \] 
\[ z_{k,l}(1)=\sigma _{N+1-k}  \left(  \left(\frac{1-x_j}{x_j}\right)_{j\neq l} \right)-\sigma _{N-k}  \left( \left(\frac{1-x_j}{x_j}\right)_{j\neq l} \right).\]
\end{lem}\qed

\subsection{Determination of the coefficients  $d_k (x_i)$ and $e_k (x_i)$.}

\begin{lem}\label{akxi}
If  $p=x_i$ for some $i=1,\dots ,N$ then 
\be d_k(p)=\frac{x_1\cdot \ldots \cdot x_N}{{p}^{N+2}}\left[\sigma _{N+1-k}\left( p-1;\left(\frac{p-x_j}{x_j}\right) _{j\neq i}\right) \right. \ee
\[ -\left. \sigma _{N+2-k}\left(p-1;\left(\frac{p-x_j}{x_j}\right) _{j\neq i}\right) \right] , \quad k=0, \ldots , N+2 \]

In particular $d_0(x_1,\dots,x_N)=0.$
\end{lem}
\begin{proof}
Follows from Lemma \ref{ak1}.
\end{proof}
\begin{lem}\label{bk1xi}
For $p=x_i, i=1,\dots , N$   the coefficients $e_k(x_i)$ are given by the formula (\ref{bbkk}) where
\[ v_k(x_i)= \sigma _{N+1-k} \left(x_i-1, \left(\frac{x_i-x_j}{x_j}\right)_{j\neq i} \right)  ,\]
\[ w_k(x_i)=\sigma _{N+1-k} \left( \left(\frac{x_i-x_j}{x_j}\right)_{j=1, \ldots ,N} \right)-\sigma _{N-k} \left(\left(\frac{x_i-x_j}{x_j}\right)_{j\neq i} \right) , \] 
\[ z_{k,l}(p)=\sigma _{N+1-k}  \left( x_i-1,  \left(\frac{x_i-x_j}{x_j}\right)_{j\neq i,l} \right)-\sigma _{N-k}  \left( x_i-1,  \left(\frac{x_i-x_j}{x_j}\right)_{j\neq i,l} \right).\]
\end{lem}
\begin{proof}
Follows from Lemma \ref{bk1}.
\end{proof}


{}

\end{document}